\documentclass[11pt,oneside,onecolumn]{amsart}
\usepackage[left=1.25in,right=1.25in,top=1.5in,bottom=1in]{geometry}

\usepackage{amsmath,amssymb,amsthm,mathrsfs,enumitem,xcolor,tikz}
\usepackage{hyperref}

\newtheoremstyle{mythm}
{.5\baselineskip}	
{.5\baselineskip}	
{}		
{}		
{\bf}	
{.}		
{ }		
{}		

\theoremstyle{mythm}
\newtheorem{theorem}{Theorem}	
\newtheorem{lemma}[theorem]{Lemma}
\newtheorem{proposition}[theorem]{Proposition}
\newtheorem{corollary}[theorem]{Corollary}
\newtheorem{definition}[theorem]{Definition}
\newtheorem{example}[theorem]{Example}
\newtheorem{remark}[theorem]{Remark}

\newtheorem{question}{Question}

\title{Translation Results for some Selection Games with Minimal Cusco Maps}

\author{Christopher Caruvana}
\address{School of Sciences\\
Indiana University Kokomo\\
2300 S. Washington Street, Kokomo, IN 46902 USA}
\email{chcaru@iu.edu}
\urladdr{https://chcaru.pages.iu.edu/}

\author{Jared Holshouser}
\address{Independent Researcher}
\email{jholshouser1321@gmail.com}
\urladdr{https://jaredholshouser.github.io/}

\date{\today}

\begin{document}

\begin{abstract}
  We establish relationships between various topological selection games involving the space of minimal cusco maps into the real line and the underlying domain.
  These connections occur across different topologies, including the topology of pointwise convergence and the topology of uniform convergence on compacta.
  Full and limited-information strategies are investigated.
  The primary games we consider are Rothberger-like games, generalized point-open games, strong fan-tightness games, Tkachuk's closed discrete selection game,
  and Gruenhage's \(W\)-games.
  We also comment on the difficulty of generalizing the given results to other classes of functions.
\end{abstract}

\maketitle

\section{Introduction}

Minimal upper-semicontinuous compact-valued functions have a rich history. The topic can be traced back to the study of holomorphic functions and cluster sets, see \cite{ClusterSetsBook}.
The phrase \emph{minimal usco} was coined by Christensen \cite{Christensen1982}, where a topological game similar to the Banach-Mazur game was considered.
When the codomain is a linear space, the term \emph{cusco map} refers to usco maps which are convex-valued.
Usco and cusco maps have been objects of study since they provide insights into the underlying topological properties of the convex subdifferential and the Clarke generalized gradient \cite{BorweinZhu}.
In this paper, using some techniques similar to those of Hol{\'{a}} and Hol{\'{y}} \cite{HolaHoly2022,HolaHoly2023},
we tie connections between a space \(X\) and the space of minimal cusco maps with the topology of uniform convergence on
certain kinds of subspaces of \(X\). The results are analogous to those of of \cite{CCUsco} and similar in spirit to those appearing in \cite{ClontzHolshouser,CHCompactOpen,CHContinuousFunctions}; in particular,
most of the results come in the form of selection game equivalences or dualities,
which rely on a variety of game-related results, referenced when needed.
In many contexts of interest, we see that, pertaining to the properties investigated herein, the spaces of minimal usco maps, minimal cusco maps, and continuous real-valued functions behave similarly.

Consequences of these results include Corollary \ref{cor:ParticularMetrizability}, which captures \cite[Cor 4.5]{HolaHoly2023}: a space
\(X\) is hemicompact if and only if \(\mathrm{MC}_k(X)\), the space of minimal cusco maps into \(\mathbb R\) on \(X\) with the topology of uniform convergence on compact subsets, is metrizable.
Corollary \ref{cor:ParticularMetrizability} also shows that \(X\) is hemicompact if and only if \(\mathrm{MC}_k(X)\) is not discretely selective.
Corollary \ref{cor:BigRothberger} contains the assertion that \(X\) is \(k\)-Rothberger if and only if \(\mathrm{MC}_k(X)\) has strong countable fan-tightness at \(\mathbf 0\), the constant \(\{0\}\) function.

\section{Preliminaries}

We use the word \emph{space} to mean \emph{topological space}.
Any undefined notions and terminologies are as in \cite{Engelking} or \cite{LinearTopologicalSpaces}.
Unless otherwise stated, all spaces considered are assumed to be Hausdorff.
When the parent space is understood from context, we use the notation
\(\mathrm{int}(A)\) and \(\mathrm{cl}(A)\) for
the interior and closure of \(A\), respectively.
If we must specify the topological space \(X\), we use \(\mathrm{int}_X(A)\) and
\(\mathrm{cl}_X(A)\).

Given a function \(f : X \to Y\), we denote the graph of \(f\) by
\(\mathrm{gr}(f) = \{ \langle x , f(x) \rangle : x \in X \}\).
For a set \(X\), we let \(\wp(X)\) denote the set of subsets of \(X\) and
\(\wp^+(X) = \wp(X) \setminus \{\emptyset\}\).
For sets \(X\) and \(Y\), we let \[\mathrm{Fn}(X,Y) = \bigcup_{A \in \wp^+(X)} Y^A;\]
that is, \(\mathrm{Fn}(X,Y)\) is the collection of all \(Y\)-valued functions
defined on non-empty subsets of \(X\).

When a set \(X\) is implicitly serving as the parent space in context,
given \(A \subseteq X\), we will let \(\mathbf{1}_A\) be the indicator (or characteristic) function for \(A\).

For any set \(X\), we let \(X^{<\omega}\) denote the set of finite sequences of \(X\) and \([X]^{<\omega}\) denote the set of finite subsets of \(X\).
For a space \(X\), we let \(K(X)\) denote the set of all non-empty compact subsets of \(X\).
We let \(\mathbb K(X)\) denote the set \(K(X)\) endowed with the Vietoris topology;
that is, the topology with basis consisting of sets of the form
\[[U_1,U_2,\ldots, U_n]
= \left\{ K \in \mathbb K(X) : K \subseteq \bigcup_{j=1}^n U_j \wedge K^n \cap \prod_{j=1}^n U_j \neq \emptyset \right\}.\]
For more about this topology, see \cite{MichaelSubsets}.

A family \(\mathcal A\) of subsets of a set \(X\) is a \emph{bornology} \cite{BornologyBook} if
\(X = \bigcup \mathcal A\), \(A \cup B \in \mathcal A\) for all \(A, B \in \mathcal A\), and,
for each \(A \in \mathcal A\), \(B \subseteq A \implies B \in \mathcal A\).
We will be interested in certain kinds of bases for bornologies, which we refer to as ideals of closed sets
since the conditions for being a bornology are similar to those of ideals.
\begin{definition}
    For a space \(X\), we say that a family \(\mathcal A \subseteq \wp^+(X)\) of closed sets is an
    \emph{ideal of closed sets} if
    \begin{itemize}
      \item
      for \(A , B \in \mathcal A\), \(A \cup B \in\mathcal A\);
      \item
      for \(A \in \mathcal B\), if \(B \subseteq A\) is closed, then \(B \in \mathcal A\); and
      \item
      for every \(x \in X\), \(\{x\} \in \mathcal A\).
    \end{itemize}
\end{definition}

Throughout, we will assume that any ideal of closed sets under consideration doesn't contain
the entire space \(X\).
Two ideals of closed sets of primary interest are
the collection of non-empty finite subsets of an infinite space \(X\) and
the collection of non-empty compact subsets of a non-compact space \(X\).

\subsection{Selection Games}
Topological games have a long history, much of which can be gathered from Telg{\'a}rsky's survey
\cite{TelgarskySurvey}.
In this paper, we will be dealing only with single-selection games of countable length.

\begin{definition}
    Given sets \(\mathcal A\) and \(\mathcal B\), we define the \emph{single-selection game}
    \(\mathsf G_1(\mathcal A, \mathcal B)\) as follows.
    For each \(n\in\omega\), One chooses \(A_n \in \mathcal A\) and Two responds with \(x_n \in A_n\).
    Two is declared the winner if \(\{ x_n : n \in \omega \} \in \mathcal B\).
    Otherwise, One wins.
\end{definition}

The study of games naturally inspires questions about the existence of various kinds of strategies.
Infinite games and corresponding full-information strategies were both introduced in \cite{GaleStewart}.
Some forms of limited-information strategies came shortly after, like positional (also known as stationary) strategies \cite{DavisGames,SchmidtGames}.
For more on stationary and Markov strategies, see \cite{GalvinTelgarsky}.
For the strategy types to be used in this paper, along with associated notation, we refer the reader to
\cite[Def 1.3]{CCUsco}.

Selection games and selection principles are intimately related.
For more details on selection principles and relevant references,
see \cite{ScheepersI,KocinacSelectedResults}.

Since this paper will focus on single-selection games of a countable length,
we only recall single-selection principles of a countable length.
\begin{definition}
    Let \(\mathcal A\) and \(\mathcal B\) be collections.
    The \emph{single-selection principle} \(\mathsf{S}_1(\mathcal A, \mathcal B)\) for a space \(X\) is the following property.
    Given any \(A \in \mathcal A^\omega\), there exists \(\vec x \in \prod_{n\in\omega} A_n\) such that
    \(\{ \vec{x}_n : n \in \omega \} \in \mathcal B\).
\end{definition}
As mentioned in \cite[Prop 15]{ClontzDualSelection}, \(\mathsf{S}_1(\mathcal A, \mathcal B)\) holds if and only if
\(\mathrm{I} {\not\uparrow}_{\mathrm{pre}}\mathsf{G}_1(\mathcal{A},\mathcal{B})\).
Hence, we may establish equivalences between certain selection principles by addressing the corresponding
selection games.

\begin{definition}
  For a space \(X\), an open cover \(\mathscr U\) of \(X\) is said to be \emph{non-trivial}
  if \(\emptyset \not\in \mathscr U\) and \(X \not\in \mathscr U\).
\end{definition}
Common cover types that appear in selection principle theory are \(\omega\)-covers and \(k\)-covers, whose definitions will be recalled below.
These cover types are generalized in the following definition.
\begin{definition}
  Let \(X\) be a space and \(\mathcal A\) be a set of closed subsets of \(X\).
  We say that a non-trivial cover \(\mathscr U\) of \(X\) is an \emph{\(\mathcal A\)-cover}
  if, for every \(A \in \mathcal A\), there exists \(U \in \mathscr U\) so that
  \(A \subseteq U\).
\end{definition}

Though \(\mathcal A\)-covers were used in \cite{CHCompactOpen} with the
notation \(\mathcal O(X,\mathcal A)\) referring to the collection of all such covers,
these covers didn't receive the name of \(\mathcal A\)-covers until \cite{CHHyperspaces}.
Also, \(\mathcal A\)-covers were independently defined as \(\omega_{\mathcal I}\)-covers and studied in \cite{LopezCallejasCruzCastillo}
where the authors investigate Ramsey-like properties.
However, these notions are not new and the essential idea appears as early as 1975 in \cite{Telgarsky1975}
where Telg{\'{a}}sky defines \(\mathbf K\)-covers relative to any collection of sets \(\mathbf K\).

For a collection \(\mathcal A\), we use \(\neg \mathcal A\) to denote the collection
of sets which are not in \(\mathcal A\).
Throughout this paper, we employ commonly used notation for cover types and other collections
of topological objects.
For an explicit list, we refer the reader to \cite[Def 1.7]{CCUsco}.

Note that, in our notation, \(\mathcal O_X([X]^{<\omega})\) is the set of all \(\omega\)-covers of \(X\),
which we will denote by \(\Omega_X\), and that \(\mathcal O_X(K(X))\) is the set of all \(k\)-covers of \(X\),
which we will denote by \(\mathcal K_X\).
We also use \(\Gamma_\omega(X)\) to denote \(\Gamma_X([X]^{<\omega})\) and \(\Gamma_k(X)\) to denote
\(\Gamma_X(K(X))\).

Note that \(\mathsf{S}_{1}(\mathcal O_X, \mathcal O_X)\) is the Rothberger property
and \(\mathsf G_1(\mathcal O_X, \mathcal O_X)\) is the Rothberger game.
If we let \(\mathbb P_X = \{ \mathscr N_{X,x} : x \in X \}\), then
\(\mathsf G_1(\mathbb P_X, \neg \mathcal O)\) is a rephrasing of the point-open game studied by Galvin \cite{Galvin1978}
and Telg{\'{a}}rsky \cite{Telgarsky1975}.
The games \(\mathsf G_1(\mathscr N_{X,x} , \neg \Gamma_{X,x})\) and \(\mathsf G_1(\mathscr N_{X,x} , \neg \Omega_{X,x})\)
are two variants of Gruenhage's \(W\)-game (see \cite{Gruenhage1976}).
We refer to \(\mathsf G_1(\mathscr N_{X,x} , \neg \Gamma_{X,x})\) as Gruenhage's converging \(W\)-game
and \(\mathsf G_1(\mathscr N_{X,x} , \neg \Omega_{X,x})\) as Gruenhage's clustering \(W\)-game.
The games \(\mathsf G_1(\mathscr T_X, \neg \Omega_{X,x})\) and \(\mathsf G_1(\mathscr T_X, \mathrm{CD}_X)\) were introduced
by Tkachuk (see \cite{Tkachuk2018,TkachukTwoPoint}) and tied to Gruenhage's \(W\)-games in \cite{TkachukTwoPoint,ClontzHolshouser}.
The strong countable dense fan-tightness game at \(x\) is \(\mathsf G_1(\mathscr D_X , \Omega_{X,x} )\)
and the strong countable fan-tightness game at \(x\) is \(\mathsf G_1(\Omega_{X,x} , \Omega_{X,x} )\) (see \cite{BarmanDown2011}).

Since we work in contexts which include both full- and limited-information strategies, we incorporate this context into our usage of the game relations
of \emph{equivalence} and \emph{duality}.
For the explicit details, we refer the reader to Definitions 1.9 and 1.11 in \cite{CCUsco}.
Accordingly, we will use the notation \(\mathcal G \equiv \mathcal H\) to denote that two selection games \(\mathcal G\)
and \(\mathcal H\) are equivalent.
\begin{definition}[{\cite{CHContinuousFunctions}}]
  Given selection games \(\mathcal G\) and \(\mathcal H\), we say that \(\mathcal G \leq_{\mathrm{II}} \mathcal H\) if the following implications hold:
  \begin{itemize}
    \item
    \(\mathrm{II} \uparrow_{\mathrm{mark}} \mathcal G \implies \mathrm{II} \uparrow_{\mathrm{mark}} \mathcal H\)
    \item
    \(\mathrm{II} \uparrow \mathcal G \implies \mathrm{II} \uparrow \mathcal H\)
    \item
    \(\mathrm{I} \not\uparrow \mathcal G \implies \mathrm{I} \not\uparrow \mathcal H\)
    \item
    \(\mathrm{I} {\not\uparrow}_{\mathrm{pre}} \mathcal G \implies \mathrm{I} {\not\uparrow}_{\mathrm{pre}} \mathcal H\)
  \end{itemize}
\end{definition}
Note that \(\leq_{\mathrm{II}}\) is transitive and that if \(\mathcal G \leq_{\mathrm{II}} \mathcal H\) and \(\mathcal H \leq_{\mathrm{II}} \mathcal G\), then \(\mathcal G \equiv \mathcal H\).
For another important comment on how these notions of equivalence and duality interact, see \cite[Lemma 1.12]{CCUsco}.

Certain general game dualities, as corollaries to \cite[Cor 26]{ClontzDualSelection}, which will be used below are
recorded in \cite[Lemma 1.13]{CCUsco}.

We now state the translation theorem we will be using to establish some game equivalences.
\begin{theorem}[{\cite[Theorem 12]{CHContinuousFunctions}}]\label{TranslationTheorem}
    Let \(\mathcal A\), \(\mathcal B\), \(\mathcal C\), and \(\mathcal D\) be collections.
    Suppose there are functions \(\overleftarrow{T}_{\mathrm{I},n}:\mathcal B \to \mathcal A\) and
    \(\overrightarrow{T}_{\mathrm{II},n}: \left( \bigcup \mathcal A \right) \times \mathcal B \to \bigcup \mathcal B\)
    for each \(n \in \omega\), such that
    \begin{enumerate}[label=(\roman*)]
        \item
        if \(x \in \overleftarrow{T}_{\mathrm{I},n}(B)\), then \(\overrightarrow{T}_{\mathrm{II},n}(x,B) \in B\), and
        \item
        if \(\langle x_n : n \in \omega \rangle \in \prod_{n\in\omega} \overleftarrow{T}_{\mathrm{I},n}(B_n)\)
        and \(\{x_n : n\in\omega\} \in \mathcal C\),
        then \(\left\{\overrightarrow{T}_{\mathrm{II},n}(x_n,B_n) : n \in \omega \right\} \in \mathcal D\).
    \end{enumerate}
    Then \(\mathsf G_1(\mathcal A,\mathcal C) \leq_{\mathrm{II}} \mathsf G_1(\mathcal B, \mathcal D)\).
\end{theorem}
This translation theorem essentially automates the work one would do to prove the connections between two games at all the different levels of strategy. For a quick example of how this translation result gets used, when we establish a connection between the open covers of a space \(X\) and the continuous real-valued functions on \(X\), we can have \(\overleftarrow{T}_{\mathrm{I},n}\) take in a sequence of functions, and output the cover formed by their preimages of a neighborhood around \(0\), and \(\overrightarrow{T}_{\mathrm{II},n}\) would take in an open set and a sequence of continuous functions, and output a function from the sequence whose preimage is that open set.

We will also need a separation axiom for some results in the sequel.
\begin{definition}
    Let \(X\) be a space and let \(\mathcal A\) be an ideal of closed subsets of \(X\).
    We say that \(X\) is \emph{\(\mathcal A\)-normal} if, given any \(A \in \mathcal A\)
    and \(U \subseteq X\) open with \(A \subseteq U\), there exists an open set \(V\)
    so that \(A \subseteq V \subseteq \mathrm{cl}(V) \subseteq U\).
\end{definition}
Note that, if \(X\) is \(\mathcal A\)-normal, then \(X\) is regular.
If \(\mathcal A = K(X)\) and \(X\) is regular, then \(X\) is \(\mathcal A\)-normal.

\subsection{Uniform Spaces} \label{section:UniformSpaces}

For all necessary and relevant background regarding uniformities and uniform spaces,
see \cite[Chapter 6]{Kelley} and \cite[{\S}8.1]{Engelking}.

We will use the following key fact about uniform spaces below.
\begin{theorem}[see \cite{Kelley}] \label{thm:UniformMetrizability}
    A Hausdorff uniform space \((X,\mathcal E)\) is metrizable if and only if
    \(\mathcal E\) has a countable base.
\end{theorem}

For a uniform space \((X,\mathcal E)\), there is a natural way to define a uniformity on \(K(X)\)
which is directly analogous to the Pompeiu-Hausdorff distance defined in the context of metric spaces.
\begin{definition}
    Let \((X,\mathcal E)\) be a uniform space and, for \(E \in \mathcal E\), define
    \[hE = \{\langle K , L \rangle \in K(X)^2 : K \subseteq E[L] \wedge L \subseteq E[K]\}.\]
\end{definition}

Just as the Pompeiu-Hausdorff distance on compact subsets generates the Vietoris topology,
the analogous uniformity also generates the Vietoris topology.
\begin{theorem}[see {\cite[Chapter 2]{CastaingValadier1977}}] \label{thm:UniformVietoris}
    For a uniform space \((X,\mathcal E)\), \(\mathcal B = \{ hE : E \in \mathcal E \}\)
    is a base for a uniformity on \(K(X)\);
    the topology generated by the uniform base \(\mathcal B\) is the Vietoris topology.
\end{theorem}

For the set of functions from a space \(X\) to a uniform space \((Y,\mathcal E)\),
we review, following \cite[Chapter 7]{Kelley}, the uniformity which generates the topology of uniform convergence on a family
of subsets of \(X\).
\begin{definition} \label{def:UniformStructure}
    For the set \(Y^X\) of functions from a set \(X\) to a uniform space \((Y,\mathcal E)\), we define,
    for \(A \in \wp^+(X)\) and \(E \in \mathcal E\),
    \[
        \mathbf U(A,E) = \{\langle f , g \rangle \in (Y^X)^2 : (\forall x \in A)\ \langle f(x), g(x) \rangle \in E\}.
    \]
    For the set of functions \(X \to \mathbb K(Y)\), we let \(\mathbf W(A,E) = \mathbf U(A,hE)\).
\end{definition}
If \(\mathcal B\) is a base for a uniformity on \(Y\) and \(\mathcal A\) is an ideal of subsets of \(X\),
then \(\{ \mathbf U(A,B) : A \in \mathcal A , B \in \mathcal B \}\) forms a base for a uniformity on \(Y^X\).
The corresponding topology generated by this base for a uniformity is the topology of uniform
convergence on \(\mathcal A\).
Consequently, \(\{ \mathbf W(A,B) : A \in \mathcal A , B \in \mathcal B \}\) is a base for a uniformity
on \(\mathbb K(Y)^X\).

\subsection{Usco and Cusco Mappings}

In this section, we introduce the basic facts of usco and cusco mappings needed for this paper.
For a thorough introduction to usco mappings, see \cite{UscoBook}.
Of primary use are Theorems \ref{thm:HolaHolyUscoChar} and \ref{thm:HolaHolyCuscoChar} which offer convenient
characterizations of minimal usco and cusco maps, respectively.

A \emph{set-valued} function from \(X\) to \(Y\) is a function \(\Phi : X \to \wp(Y)\).
These are sometimes also referred to as \emph{multi-functions}.
\begin{definition}
    A set-valued function \(\Phi : X \to \wp(Y)\) is said to be \emph{upper semicontinuous} if, for every open
    \(V \subseteq Y\), \[\Phi^\leftarrow(V) := \{ x \in X : \Phi(x) \subseteq V\}\] is open in \(X\).
    An \emph{usco} map from a space \(X\) to \(Y\) is a set-valued map \(\Phi\) from \(X\) to \(Y\) which is upper semicontinuous
    and whose range is contained in \(\mathbb K(Y)\).
    Let \(\mathrm{USCO}(X,Y)\) denote the collection of all usco maps \(X \to \mathbb K(Y)\).
    
    An usco map \(\Phi : X \to \mathbb K(Y)\) is said to be \emph{minimal} if its graph is minimal with respect to the
    \(\subseteq\) relation.
    Let \(\mathrm{MU}(X,Y)\) denote the collection of all minimal usco maps \(X \to \mathbb K(Y)\).
\end{definition}
\begin{definition}
    Suppose \(Y\) is a Hausdorff linear space.
    An usco map \(\Phi : X \to \mathbb K(Y)\) is said to be \emph{cusco} if \(\Phi(x)\) is convex
    for every \(x \in X\).
    A cusco map \(\Phi : X \to \mathbb K(Y)\) is said to be \emph{minimal} if its graph is minimal with respect to the
    \(\subseteq\) relation.
    Let \(\mathrm{MC}(X,Y)\) denote the collection of all minimal cusco maps \(X \to \mathbb K(Y)\).
\end{definition}

It is clear that any continuous \(\Phi : X \to \mathbb K(Y)\) is usco and
that there are continuous \(\Phi : X \to \mathbb K(Y)\) which are not minimal.
As demonstrated by \cite[Example 1.31]{CCUsco}, there are minimal usco maps \(\mathbb R \to \mathbb K(\mathbb R)\)
which are not continuous.
We will show, in Example \ref{ex:UscoNotCont}, that such an example can be extended to produce a
minimal cusco map \(\mathbb R \to \mathbb K(\mathbb R)\) that is not continuous.

\begin{definition}
    Suppose \(\Phi : X \to \wp^+(Y)\).
    We say that a function \(f : X \to Y\) is a \emph{selection} of \(\Phi\)
    if \(f(x) \in \Phi(x)\) for every \(x \in X\).
    We let \(\mathrm{sel}(\Phi)\) be the set of all selections of \(\Phi\).
    
    If \(D \subseteq X\) is dense and \(f : D \to Y\) is so that \(f(x) \in \Phi(x)\)
    for each \(x \in D\), we say that \(f\) is a \emph{densely defined selection}
    of \(\Phi\).
\end{definition}

The notion of subcontinuity was introduced by Fuller \cite{Fuller1968} which can be extended
to so-called densely defined functions in the following way.
\begin{definition}
    Suppose \(D \subseteq X\) is dense.
    We say that a function \(f : D \to Y\) is \emph{subcontinuous} if, for every \(x \in X\)
    and every net \(\langle x_\lambda : \lambda \in \Lambda \rangle\) in \(D\) with \(x_\lambda \to x\),
    \(\langle f(x_\lambda) : \lambda \in \Lambda \rangle\) has an accumulation point.
\end{definition}

The following is a well-known property of usco maps that will be used in this paper.
\begin{lemma} \label{lem:SelectionsAreSubcontinuous}
    If \(\Phi \in \mathrm{USCO}(X,Y)\), then any selection of \(\Phi\) is subcontinuous.
\end{lemma}

The notion of semi-open sets was introduced by Levine \cite{Levine1963}.
\begin{definition}
    For a space \(X\), a set \(A \subseteq X\) is said to be \emph{semi-open} if
    \(A \subseteq \mathrm{cl}\ \mathrm{int}(A)\).
\end{definition}

The notion of quasicontinuity was introduced by Kempisty \cite{Kempisty1932}
and surveyed by Neubrunn \cite{Neubrunn}.
\begin{definition}
    A function \(f : X \to Y\) is said to be \emph{quasicontinuous} if, for each open
    \(V \subseteq Y\), \(f^{-1}(V)\) is semi-open in \(X\).
    
    If \(D \subseteq X\) is dense and \(f : D \to Y\), we will say that \(f\) is quasicontinuous
    if it is quasicontinuous on \(D\) with the subspace topology.
\end{definition}

\begin{definition}
    For \(f \in \mathrm{Fn}(X,Y)\), define \(\overline{f} : X \to \wp(Y)\) by the rule
    \[\overline{f}(x) = \{ y \in Y : \langle x , y \rangle \in \mathrm{cl}\ \mathrm{gr}(f) \}.\]
\end{definition}
For a linear space \(Y\) and \(A \subseteq Y\), we will use \(\mathrm{conv}{\,}A\) to denote
the convex hull of \(A\).
\begin{definition}
    Suppose \(Y\) is a Hausdorff locally convex linear space.
    For \(f \in \mathrm{Fn}(X,Y)\), define \(\check{f} : X \to \wp(Y)\) by the rule
    \[\check{f}(x) = \mathrm{cl}{\,}\mathrm{conv}{\,}\overline{f}(x).\]
\end{definition}

\begin{theorem}[Hol{\'{a}}, Hol{\'{y}} \cite{HolaHoly2009,HolaHoly2014}] \label{thm:HolaHolyUscoChar}
    Suppose \(Y\) is regular and that \(\Phi : X \to \wp^+(Y)\).
    Then the following are equivalent:
    \begin{enumerate}[label=(\roman*)]
        \item \label{tfae:MinimalUsco}
        \(\Phi\) is minimal usco.
        \item \label{tfae:AllSubQuasi}
        Every selection \(f\) of \(\Phi\) is subcontinuous, quasicontinuous, and \(\Phi = \overline{f}\).
        \item \label{tfae:SomeSubQuasi}
        There exists a selection \(f\) of \(\Phi\) which is subcontinuous,
        quasicontinuous, and \(\Phi = \overline{f}\).
        \item \label{tfae:SomeDenseSubQuasi}
        There exists a densely defined selection \(f\) of \(\Phi\) which is subcontinuous,
        quasicontinuous, and \(\Phi = \overline{f}\).
    \end{enumerate}
\end{theorem}

Important initial contributions to the following characterization of cusco maps are found
in \cite{GilesMoors,BorweinZhu}.

\begin{theorem}[Hol{\'{a}}, Hol{\'{y}} \cite{HolaHoly2014}] \label{thm:HolaHolyCuscoChar}
    Suppose \(Y\) is a Hausdorff locally convex linear space
    in which the closed convex hull of a compact set is compact and that \(\Phi : X \to \wp^+(Y)\).
    Then the following are equivalent:
    \begin{enumerate}[label=(\roman*)]
        \item
        \(\Phi\) is minimal cusco.
        \item
        There exists a selection \(f\) of \(\Phi\) which is subcontinuous,
        quasicontinuous, and \(\Phi = \check{f}\).
        \item
        There exists a densely defined selection \(f\) of \(\Phi\) which is subcontinuous,
        quasicontinuous, and \(\Phi = \check{f}\).
    \end{enumerate}
\end{theorem}
\begin{remark} \label{rmk:MaximumSelection}
    Based on \cite[Theorem 3.2]{HolaHoly2014}, for each multi-function \(\Phi \in \mathrm{MC}(X,\mathbb R)\),
    the function \(f:X \to \mathbb R\) defined by \(f(x) = \max \Phi(x)\) is subcontinuous and quasicontinuous.
\end{remark}

The following is motivated by Theorem \ref{thm:HolaHolyCuscoChar}.
\begin{definition}
    For any \(\Phi \in \mathrm{MC}(X,Y)\), let \(\mathrm{sel}_{QS}(\Phi)\) be the collection
    of selections \(f\) of \(\Phi\) that are quasicontinuous, subcontinuous, and \(\check{f} = \Phi\).
\end{definition}

As suggested by Theorems \ref{thm:HolaHolyUscoChar} and \ref{thm:HolaHolyCuscoChar},
there is a close relationship between minimal usco and minimal cusco maps.
\begin{theorem}[{\cite[Theorem 2.6]{HolaHolyRelations}}] \label{thm:UscoCuscoCorrespondence}
    Let \(X\) be a space and \(Y\) be a Hausdorff locally convex linear space in which the closed
    convex hull of a compact set is compact.
    Define \(\varphi : \mathrm{MU}(X,Y) \to \mathbb K(Y)^X\) by the rule
    \(\varphi(F)(x) = \mathrm{cl}{\,}\mathrm{conv} F(x).\)
    Then \(\varphi\) is a bijection  from \(\mathrm{MU}(X,Y)\) to \(\mathrm{MC}(X,Y)\).
\end{theorem}
In fact, by \cite[Theorem 3.1]{HolaHolyRelations}, when \(Y\) is a Banach space and
\(\mathrm{MU}(X,Y)\) and \(\mathrm{MC}(X,Y)\) are both given the topology of either pointwise
convergence or uniform convergence on compacta, \(\varphi\) is continuous.
When \(X\) is locally compact, \(Y\) is a Banach space, and the spaces
\(\mathrm{MU}(X,Y)\) and \(\mathrm{MC}(X,Y)\) are given the topology of uniform convergence
on compacta, by \cite[Theorem 3.2]{HolaHolyRelations}, \(\varphi\) is
a homeomorphism.
However, it seems unknown whether \(\mathrm{MU}(X,Y)\) and \(\mathrm{MC}(X,Y)\) are always homeomorphic,
though \cite[Example 3.2]{HolaHolyRelations} may lead one to conjecture that
this is not the case since it is an example where the natural mapping \(\varphi\) is not
a homeomorphism.

We will be using the following to construct certain functions.
\begin{lemma} \label{lem:UsefulFunction}
    Let \(f , g : X \to Y\) and \(U \in \mathscr T_X\) and define \(h : X \to Y\) by the rule
    \[h(x) = \begin{cases} f(x), & x \in \mathrm{cl}(U); \\ g(x), & x \not\in \mathrm{cl}(U).\end{cases}\]
    \begin{enumerate}[label=(\roman*)]
        \item \label{charSub}
        If \(f\) and \(g\) are subcontinuous, then \(h\) is subcontinuous.
        \item \label{charQuasi}
        If \(f\) is constant, and \(g\) is quasicontinuous,
        then \(h\) is quasicontinuous.
    \end{enumerate}
    Consequently, if \(f\) is constant and \(g\) is both subcontinuous and quasicontinuous, then
    \(h\) is both subcontinuous and quasicontinuous.
    Moreover, \(\overline{h}\) is minimal usco;
    under the assumption that \(Y\) is a Hausdorff locally convex linear
    space in which the closed convex hull of a compact set is compact,
    \(\check{h}\) is minimal cusco.
\end{lemma}
\begin{proof}
    All except for the fact that \(\check{h}\) is minimal cusco when \(Y\) is a Hausdorff locally convex linear
    space in which the closed convex hull of a compact set is compact is proved in \cite[Lemma 1.30]{CCUsco}.
    For the remaining remark, simply appeal to Theorem \ref{thm:HolaHolyCuscoChar}.
\end{proof}
\begin{example} \label{ex:UscoNotCont}
    Consider \(\mathrm{MC}(\mathbb R,\mathbb R)\).
    By Lemma \ref{lem:UsefulFunction}, \(\Phi := \check{\mathbf{1}}_{[0,1]}\) is minimal cusco.
    However, \(\Phi\) is not continuous since
    \[\{0,1\} = \{ x\in \mathbb R : \Phi(x) \in [(-0.5,0.75),(0.25,1.5)] \}.\]
\end{example}
Hence, when \(Y\) is metrizable, studying the space \(\mathrm{MC}(X,Y)\) is, in general, different than studying
the space of continuous functions into a metrizable space.

Although the next two results are known, we provide short proofs for the convenience of the reader.
\begin{corollary} \label{cor:PreUscoAgreeOpen}
    Suppose \(X\) is a space, \(Y\) is a Hausdorff locally convex linear space
    in which the closed convex hull of a compact set is compact,
    \(\Phi, \Psi \in \mathrm{MC}(X,Y)\), \(f \in \mathrm{sel}_{QS}(\Phi)\), \(g\in \mathrm{sel}_{QS}(\Psi)\),
    and \(f\restriction_U = g\restriction_U\).
    Then \(\Phi\restriction_U = \Psi\restriction_U\).
\end{corollary}
\begin{proof}
    Set \(\Phi^* = \overline{f}\) and \(\Psi^* = \overline{g}\). Then
    \(\Phi^*, \Psi^* \in \mathrm{MU}(X,Y)\) by Theorem \ref{thm:HolaHolyUscoChar};
    so \(\Phi^* \subseteq \Phi\) and \(\Psi^* \subseteq \Psi\).
    By \cite[Cor 1.32]{CCUsco}, \(\Phi^\ast\restriction_U = \Psi^\ast\restriction_U\).
    Finally, by Theorem \ref{thm:UscoCuscoCorrespondence}, \(\Phi\restriction_U = \Psi\restriction_U\).
\end{proof}
\begin{corollary} \label{cor:CuscoAgreeOpen}
    Suppose \(X\) is a space and \(Y\) is a Hausdorff locally convex linear space
    in which the closed convex hull of compact set is compact.
    If \(A \subseteq X\) is non-empty, \(U, V \in \mathscr T_X\) are such that
    \(A \subseteq V \subseteq \mathrm{cl}(V) \subseteq U\),
    \(\Phi \in \mathrm{MC}(X,Y)\), and \(f \in \mathrm{sel}_{QS}(\Phi)\) is such that
    \(\Phi = \check{f}\), then, for \(y_0 \in Y\), the map \(g : X \to Y\) defined by
    \[
        g(x) = \begin{cases} y_0, & x \in \mathrm{cl}(X \setminus \mathrm{cl}(V));\\
        f(x), & \mathrm{otherwise},\end{cases}
    \]
    has the property that \(\Psi := \check{g} \in \mathrm{MC}(X,Y)\), \(\langle \Phi, \Psi \rangle \in \mathbf W(A,E)\)
    for any entourage \(E\) of \(Y\), and \(g[X \setminus U] = \{y_0\}\).
\end{corollary}
\begin{proof}
    Note that Lemma \ref{lem:UsefulFunction} implies that \(g\) is subcontinuous and quasicontinuous.
    Then \(\Psi := \check{g} \in \mathrm{MC}(X,Y)\) by Theorem \ref{thm:HolaHolyCuscoChar}.

    Since \(V\) is open, \(\mathrm{cl}(X \setminus \mathrm{cl}(V)) \subseteq X \setminus V\).
    Then \(g(x) = f(x)\) for all \(x \in V\).
    By Corollary \ref{cor:PreUscoAgreeOpen}, we see that \(\Phi\restriction_A = \Psi\restriction_A\)
    as \(A \subseteq V\).
    Also, since \(X \setminus U \subseteq X \setminus \mathrm{cl}(V) \subseteq \mathrm{cl}(X \setminus \mathrm{cl}(V))\),
    we see that \(g[X \setminus U] = \{y_0\}\).
\end{proof}

Since every linear space is a topological group, we can use the uniform structure
generated by the neighborhoods of the identity;
that is, the sets of the form
\[U_R = \{ \langle x , y \rangle \in X : xy^{-1} \in U \}\]
where \(U\) is a neighborhood of the identity form a base
for a uniformity on a topological group \(X\).
We can also restrict our attention to a particular basis at the identity to produce
this uniform structure.
Note that a closed neighborhood of identity generates a closed entourage as viewed as a
subset of \(X^2\) with the product topology.

The following can be seen as a modification of \cite[Lemma 6.1]{HolaNovotny}.
We provide a full proof for the convenience of the reader.

\begin{corollary} \label{cor:AlmostLikeContinuous}
    Let \(X\) be a space and \(Y\) be a Hausdorff locally convex linear space.
    Suppose \(\Phi,\Psi \in \mathrm{MC}(X,Y)\), \(E\) is a closed convex neighborhood of \(0_Y\),
    \(D \subseteq X\) is dense, and
    \(\langle \Phi(x),\Psi(x) \rangle \in hE_R\) for all \(x \in D\).
    Then \(\langle \Phi(x), \Psi(x) \rangle \in hE_R\)
    for all \(x \in X\).
\end{corollary}
\begin{proof}
  Define \(F : X \to \wp(Y)\) by \(F(x) = E_R[\Phi(x)]\)
  and note that \(F(x)\) is convex and closed for each \(x \in X\).

  We show that the graph of \(F\) is closed.
  Suppose \(\langle x, y \rangle \in \mathrm{cl}\ \mathrm{gr}(F)\)
  and let \(\langle \langle x_\lambda , y_\lambda \rangle : \lambda \in \Lambda \rangle\)
  be a net in \(\mathrm{gr}(F)\) so that \(\langle x_\lambda, y_\lambda \rangle \to \langle x , y \rangle\).
  Since \(y_\lambda \in E_R[\Phi(x_\lambda)]\), we can let
  \(w_\lambda \in \Phi(x_\lambda)\) be so that \(y_\lambda \in E_R[w_\lambda]\).
  Observe that, since \(x_\lambda \to x\) and \(w_\lambda \in \Phi(x_\lambda)\) for each \(\lambda \in \Lambda\),
  by Lemma \ref{lem:SelectionsAreSubcontinuous},
  \(\langle w_\lambda : \lambda \in \Lambda \rangle\) has an accumulation point \(w \in \Phi(x)\).
  Since \(y_\lambda \to y\) and \(w\) is an accumulation point of \(\langle w_\lambda : \lambda \in \Lambda \rangle\),
  \(\langle w, y\rangle\) is an accumulation point of
  \(\langle \langle w_\lambda , y_\lambda \rangle : \lambda \in \Lambda \rangle\).
  Moreover, as \(\langle w_\lambda , y_\lambda \rangle \in E_R\) for all \(\lambda\in\Lambda\)
  and \(E_R\) is closed, we see that \(\langle w , y \rangle \in E_R\).
  Hence, \(y \in E_R[w] \subseteq E_R[\Phi(x)] = F(x)\).
  That is, \(\langle x, y \rangle \in \mathrm{gr}(F)\) which establishes that
  \(\mathrm{gr}(F)\) is closed.

  By Theorem \ref{thm:HolaHolyCuscoChar}, we can let \(g : D \to Y\) be subcontinuous
  and quasicontinuous so that \(g(x) \in \Psi(x)\) for each \(x\in D\) and \(\Psi = \check{g}\).
  Since \(\mathrm{gr}(F)\) is closed, convex-valued, and \(\mathrm{gr}(g) \subseteq \mathrm{gr}(F)\),
  we see that \(\mathrm{cl}\ \mathrm{gr}(g) \subseteq \mathrm{gr}(F)\)
  and
  \(\check g(x) \subseteq F(x)\).
  That is, \(\Psi(x) \subseteq F(x) = E_R[\Phi(x)]\) for all \(x \in X\).

  A symmetric argument shows that \(\Phi(x) \subseteq E_R[\Psi(x)]\) for all \(x \in X\),
  finishing the proof.
\end{proof}

\section{Results}

For the remainder of the paper, we will be interested only in real set-valued functions;
so we will let \(\mathrm{MU}(X) = \mathrm{MU}(X,\mathbb R)\) and \(\mathrm{MC}(X) = \mathrm{MC}(X,\mathbb R)\).
We also use, for \(\varepsilon > 0\), \[\Delta_\varepsilon = \{ \langle x, y \rangle \in \mathbb R^2 : |x-y| < \varepsilon \}.\]
For \(A \subseteq X\), we will use \(\mathbf U(A,\varepsilon) = \mathbf U(A,\Delta_\varepsilon)\)
and \(\mathbf W(A,\varepsilon) = \mathbf W(A , \Delta_\varepsilon)\).
For \(Y \subseteq \mathbb R\), let \(\mathbb B(Y,\varepsilon) = \bigcup_{y \in Y} B(y,\varepsilon)\) and
note that
\[
    \mathbf W(A, \varepsilon)
     = \left\{ \langle \Phi , \Psi \rangle \in \mathbb K(\mathbb R)^X : (\forall x \in A)
    [\Phi(x) \subseteq \mathbb B(\Psi(x),\varepsilon) \wedge \Psi(x) \subseteq \mathbb B(\Phi(x),\varepsilon)] \right\}.
\]
For \(\Phi : X \to \mathbb K(Y)\), \(A \subseteq X\), and \(\varepsilon > 0\), we let \([\Phi; A , \varepsilon]
= \mathbf W(A,\varepsilon)[\Phi]\).

Then, if \(\mathcal A\) is an ideal of closed subsets of \(X\),
we will use \(\mathrm{MU}_{\mathcal A}(X)\) (resp. \(\mathrm{MC}_{\mathcal A}(X)\)) to denote the set \(\mathrm{MU}(X)\) (resp. \(\mathrm{MC}(X)\))
with the topology generated by the base
for a uniformity \(\{ \mathbf W(A,\varepsilon) : A \in\mathcal A, \varepsilon > 0\}\).
When \(\mathcal A = [X]^{<\omega}\), we use \(\mathrm{MU}_p(X)\) and \(\mathrm{MC}_p(X)\); when \(\mathcal A = K(X)\),
we use \(\mathrm{MU}_k(X)\) and \(\mathrm{MC}_k(X)\).
We will use \(\mathbf 0\) to denote the function that is constantly \(0\) when dealing
with real-valued functions and the function that is constantly \(\{0\}\) when dealing
with set-valued maps.

\begin{theorem} \label{thm:FirstEquivalence}
    Let \(X\) be regular and let \(\mathcal A\) and \(\mathcal B\) be ideals of closed
    subsets of \(X\).
    Then,
    \begin{enumerate}[label=(\roman*)]
        \item \label{thm:RothbergerFirst}
        \(\mathsf G_1(\mathcal O_X(\mathcal A), \Lambda_X(\mathcal B))
        \leq_{\mathrm{II}} \mathsf G_1(\Omega_{\mathrm{MC}_{\mathcal A}(X),\mathbf 0},\Omega_{\mathrm{MC}_{\mathcal B}(X),\mathbf 0})\),
        \item \label{thm:RothbergerSecond}
        \(\mathsf G_1(\Omega_{\mathrm{MC}_{\mathcal A}(X),\mathbf 0},\Omega_{\mathrm{MC}_{\mathcal B}(X),\mathbf 0})
        \leq_{\mathrm{II}} \mathsf G_1(\mathscr D_{\mathrm{MC}_{\mathcal A}(X)},\Omega_{\mathrm{MC}_{\mathcal B}(X),\mathbf 0})\), and
        \item \label{thm:RothbergerThird}
        if \(X\) is \(\mathcal A\)-normal,
        \(\mathsf G_1(\mathscr D_{\mathrm{MC}_{\mathcal A}(X)},\Omega_{\mathrm{MC}_{\mathcal B}(X),\mathbf 0})
        \leq_{\mathrm{II}} \mathsf G_1(\mathcal O_X(\mathcal A), \Lambda_X(\mathcal B))\).
    \end{enumerate}
    Thus, if \(X\) is \(\mathcal A\)-normal, the three games are equivalent.
\end{theorem}
\begin{proof}
  We first address \ref{thm:RothbergerFirst}.
  Fix some \(\mathscr U_0 \in \mathcal O_X(\mathcal A)\) and let \(W_{\Phi,n} = \Phi^\leftarrow[(-2^{-n},2^{-n})]\)
  for \(\Phi \in \mathrm{MC}(X)\) and \(n \in \omega\).
  Let \[\mathfrak T_n = \{ \mathscr F \in \Omega_{\mathrm{MC}_{\mathcal A}(X),\mathbf 0} : (\exists \Phi \in \mathscr F)\ W_{\Phi,n} = X \}\]
  and \(\mathfrak T_n^\star = \Omega_{\mathrm{MC}_{\mathcal A}(X),\mathbf 0} \setminus \mathfrak T_n\).
  Define \(\overleftarrow{T}_{\mathrm{I},n} : \Omega_{\mathrm{MC}_{\mathcal A}(X),\mathbf 0} \to \mathcal O_X(\mathcal A)\)
  by the rule
  \[\overleftarrow{T}_{\mathrm{I},n}(\mathscr F) =
  \begin{cases}
    \{ W_{\Phi,n} : \Phi \in \mathscr F \}, & \mathscr F \in \mathfrak T_n^\star;\\
    \mathscr U_0, & \mathrm{otherwise}.
  \end{cases}\]
  To see that \(\overleftarrow{T}_{\mathrm{I},n}\) is defined, let \(\mathscr F \in \mathfrak T_n^\star\).
  Let \(A \in \mathcal A\) be arbitrary and choose \(\Phi \in [\mathbf 0;A,2^{-n}] \cap \mathscr F\).
  It follows that \(A \subseteq W_{\Phi,n}\).
  Hence, \(\overleftarrow{T}_{\mathrm{I},n}(\mathscr F) \in \mathcal O_X(\mathcal A)\).

  We now define
  \[\overrightarrow{T}_{\mathrm{II},n} : \mathscr T_X \times \Omega_{\mathrm{MC}_{\mathcal A}(X),\mathbf 0} \to \mathrm{MC}(X)\]
  in the following way.
  For each \(\langle U ,\mathscr F \rangle \in \mathscr T_X \times \mathfrak T_n\), let
  \(\overrightarrow{T}_{\mathrm{II},n}(U,\mathscr F) \in \mathscr F\) be so that
  \(W_{\overrightarrow{T}_{\mathrm{II},n}(U,\mathscr F),n} = X\).
  For \(\langle U, \mathscr F \rangle \in \mathscr T_X \times \mathfrak T_n^\star\) so that
  \(U \in \overleftarrow{T}_{\mathrm{I},n}(\mathscr F)\), let \(\overrightarrow{T}_{\mathrm{II},n}(U,\mathscr F) \in \mathscr F\)
  be so that \(U = W_{\overrightarrow{T}_{\mathrm{II},n}(U,\mathscr F),n}\).
  For \(\langle U, \mathscr F \rangle \in \mathscr T_X \times \mathfrak T_n^\star\) so that
  \(U \not\in \overleftarrow{T}_{\mathrm{I},n}(\mathscr F)\), let \(\overrightarrow{T}_{\mathrm{II},n}(U,\mathscr F) = \mathbf 0\).
  By construction, if \(U \in \overleftarrow{T}_{\mathrm{I},n}(\mathscr F)\),
  then \(\overrightarrow{T}_{\mathrm{II},n}(U,\mathscr F) \in \mathscr F\).

  To finish this application of Theorem \ref{TranslationTheorem}, assume that we have
  \[\langle U_n : n \in \omega \rangle \in \prod_{n\in\omega} \overleftarrow{T}_{\mathrm{I},n}(\mathscr F_n)\]
  for some sequence \(\langle \mathscr F_n : n \in \omega \rangle\) of \(\Omega_{\mathrm{MC}_{\mathcal B}(X),\mathbf 0}\)
  so that \(\{ U_n : n \in \omega \} \in \Lambda_X(\mathcal B)\).
  For each \(n \in \omega\), let \(\Phi_n = \overrightarrow{T}_{\mathrm{II},n}(U_n,\mathscr F_n)\).
  Now, let \(B \in \mathcal B\) and \(\varepsilon > 0\) be arbitrary.
  Choose \(n \in \omega\) so that \(2^{-n} < \varepsilon\) and \(B \subseteq U_n\).
  If \(\mathscr F_n \in \mathfrak T_n\), then \(\Phi_n\) has the property that \(X = \Phi_n^\leftarrow[(-2^{-n},2^{-n})]\);
  hence, \(\Phi_n \in [\mathbf 0; B, \varepsilon]\).
  Otherwise, \(B \subseteq U_n = \Phi_n^\leftarrow[(-2^{-n},2^{-n})]\) which also implies that
  \(\Phi_n \in [\mathbf 0; B , \varepsilon]\).
  Thus, \(\{ \Phi_n : n \in \omega \} \in \Omega_{\mathrm{MC}_{\mathcal B}(X),\mathbf 0}\).

  \ref{thm:RothbergerSecond} holds since \(\mathscr D_{\mathrm{MC}_{\mathcal A}(X)} \subseteq
  \Omega_{\mathrm{MC}_{\mathcal A}(X),\mathbf 0}\).

  Lastly, we address \ref{thm:RothbergerThird}.
  We define
  \(\overset{\leftarrow}{T}_{\mathrm I,n} : \mathcal O_X(\mathcal A) \to \mathscr D_{\mathrm{MC}_{\mathcal A}(X)}\)
  by the rule \[\overset{\leftarrow}{T}_{\mathrm I,n}(\mathscr U)
  = \{ \Phi \in \mathrm{MC}(X) : (\exists U \in \mathscr U)(\exists f \in \mathrm{sel}_{QS}(\Phi))\ f[X \setminus U] = \{1\} \}.\]
  To see that \(\overset{\leftarrow}{T}_{\mathrm I,n}\) is defined, let
  \(\mathscr U \in \mathcal O_X(\mathcal A)\) and consider a basic open set
  \([\Phi ; A , \varepsilon ]\).
  Then let \(U \in \mathscr U\) be so that \(A \subseteq U\) and, by \(\mathcal A\)-normality,
  let \(V\) be open so that \(A \subseteq V \subseteq \mathrm{cl}(V) \subseteq U\).
  Define \(f : X \to \mathbb R\) by the rule
  \[f(x) =
  \begin{cases}
    1, & x \in \mathrm{cl}(X \setminus \mathrm{cl}(V));\\
    \max \Phi(x), & \mathrm{otherwise}.
  \end{cases}
  \]
  By Remark \ref{rmk:MaximumSelection} and Corollary \ref{cor:CuscoAgreeOpen},
  \(\check{f} \in [\Phi; A, \varepsilon] \cap \overset{\leftarrow}{T}_{\mathrm I,n}(\mathscr U)\).
  So \(\overset{\leftarrow}{T}_{\mathrm I,n}(\mathscr U) \in \mathscr D_{\mathrm{MC}_{\mathcal A}(X)}\).

  We define \(\overset{\rightarrow}{T}_{\mathrm{II},n} : \mathrm{MC}(X) \times \mathcal O_X(\mathcal A) \to \mathscr T_X\)
  in the following way.
  Fix some \(U_0 \in \mathscr T_X\).
  For \(\langle \Phi , \mathscr U \rangle \in \mathrm{MC}(X) \times \mathcal O_X(\mathcal A)\),
  if \[\{ U \in \mathscr U : (\exists f \in \mathrm{sel}_{QS}(\Phi))\ f[X \setminus U] = \{1\} \} \neq \emptyset,\]
  let \(\overset{\rightarrow}{T}_{\mathrm{II},n}(\Phi, \mathscr U) \in \mathscr U\) be so that
  there exists \(f \in \mathrm{sel}_{QS}(\Phi)\) with the property that
  \(f[X \setminus \overset{\rightarrow}{T}_{\mathrm{II},n}(\Phi, \mathscr U)] = \{1\};\)
  otherwise, let \(\overset{\rightarrow}{T}_{\mathrm{II},n}(\Phi, \mathscr U) = U_0\).
  By construction, if \(\Phi \in \overset{\leftarrow}{T}_{\mathrm I,n}(\mathscr U)\),
  then \(\overset{\rightarrow}{T}_{\mathrm{II},n}(\Phi, \mathscr U) \in \mathscr U\).

  Suppose we have
  \[\langle \Phi_n : n \in \omega \rangle \in \prod_{n\in\omega} \overset{\leftarrow}{T}_{\mathrm I,n}(\mathscr U_n)\]
  for a sequence \(\langle \mathscr U_n : n \in \omega\rangle\) of \(\mathcal O_X(\mathcal A)\)
  with the property that \(\{ \Phi_n : n \in \omega \} \in \Omega_{\mathrm{MC}_{\mathcal B}(X),\mathbf 0}\).
  For each \(n \in \omega\), let \(U_n = \overset{\rightarrow}{T}_{\mathrm{II},n}(\Phi_n, \mathscr U_n)\).
  Since \(\mathcal B\) is an ideal of sets, we need only show that \(\langle U_n : n \in \omega \rangle\)
  is a \(\mathcal B\)-cover.
  So let \(B \in \mathcal B\) be arbitrary and let \(n \in \omega\) be so that
  \(\Phi_n \in [\mathbf 0; B, 1]\).
  Then we can let \(f \in \mathrm{sel}_{QS}(\Phi_n)\) be so that \(f[X \setminus U_n] = \{1\}\).
  Since \(\Phi_n \in [\mathbf 0;B,1]\), we see that, for each \(x \in B\), \(f(x) \in \Phi_n(x) \subseteq (-1,1)\).
  Hence, \(B \cap (X \setminus U_n) = \emptyset\), which is to say that \(B \subseteq U_n\).
  So Theorem \ref{TranslationTheorem} applies.
\end{proof}
\begin{corollary} \label{cor:Rothberger}
    Let \(\mathcal A\) and \(\mathcal B\) be ideals of closed subsets of \(X\) and suppose that
    \(X\) is \(\mathcal A\)-normal.
    Then
    \begin{align*}
        \mathcal G := \mathsf G_1(\mathcal O_X(\mathcal A), \mathcal O_X(\mathcal B))
        &\equiv \mathsf G_1(\Omega_{\mathrm{MC}_{\mathcal A}(X),\mathbf 0},\Omega_{\mathrm{MC}_{\mathcal B}(X),\mathbf 0})\\
        &\equiv \mathsf G_1(\mathscr D_{\mathrm{MC}_{\mathcal A}(X)},\Omega_{\mathrm{MC}_{\mathcal B}(X),\mathbf 0}),
    \end{align*}
    \begin{align*}
        \mathcal H := \mathsf G_1(\mathscr N_X[\mathcal A], \neg \mathcal O_X(\mathcal B))
        &\equiv \mathsf G_1(\mathscr N_{\mathrm{MC}_{\mathcal A}(X),\mathbf 0}, \neg \Omega_{\mathrm{MC}_{\mathcal B}(X),\mathbf 0})\\
        &\equiv \mathsf G_1(\mathscr T_{\mathrm{MC}_{\mathcal A}(X)}, \neg \Omega_{\mathrm{MC}_{\mathcal B}(X),\mathbf 0}),
    \end{align*}
    and \(\mathcal G\) is dual to \(\mathcal H\).
\end{corollary}
\begin{proof}
    Apply Theorem \ref{thm:FirstEquivalence}, \cite[Lemma 4]{CHContinuousFunctions}, \cite[Lemma 1.12]{CCUsco} and \cite[Lemma 1.13]{CCUsco}.
\end{proof}

We offer the following comments relating this result to other structures.
In \cite[Theorem 31]{CHHyperspaces}, inspired by Li \cite{Li2016}, the game
\(\mathsf G_1(\mathcal O_X(\mathcal A) , \mathcal O_X(\mathcal B))\) is shown to be equivalent
to the selective separability game on certain hyperspaces of \(X\).
Along a similar line, \cite[Cor 14]{CHContinuousFunctions} establishes an analogous result to
Corollary \ref{cor:Rothberger} relative to continuous real-valued functions with the
corresponding topology under the assumption that \(X\) is functionally \(\mathcal A\)-normal (\cite[Def 1.15]{CCUsco}).

\begin{lemma} \label{lem:SequentiallyCompact}
    Suppose \(\Phi \in \mathrm{USCO}(X,\mathbb R)\) and that \(A \subseteq X\) is sequentially compact.
    Then \(\Phi[A]\) is bounded.
\end{lemma}
The proof is nearly identical to the one of \cite[Lemma 2.4]{CCUsco}; one need only replace the application
of \cite[Theorem 1.29]{CCUsco} with an application of Lemma \ref{lem:SelectionsAreSubcontinuous}.

\begin{theorem} \label{thm:SecondTheorem}
    Let \(\mathcal A\) and \(\mathcal B\) be ideals of closed subsets of \(X\).
    If \(X\) is \(\mathcal A\)-normal and \(\mathcal B\) consists of sequentially compact sets, then
    \[
        \mathsf G_1(\mathscr{N}_X[\mathcal A], \neg \Lambda_X(\mathcal B))
        \leq_{\mathrm{II}} \mathsf G_1(\mathscr{T}_{\mathrm{MC}_{\mathcal A}(X)}, \mathrm{CD}_{\mathrm{MC}_{\mathcal B}(X)}).
    \]
    Consequently,
    \begin{align*}
        \mathsf G_1(\mathscr N_X[\mathcal A], \neg \mathcal O_X(\mathcal B))
        &\equiv \mathsf G_1(\mathscr N_{\mathrm{MC}_{\mathcal A}(X),\mathbf 0}, \neg \Omega_{\mathrm{MC}_{\mathcal B}(X),\mathbf 0})\\
        &\equiv \mathsf G_1(\mathscr T_{\mathrm{MC}_{\mathcal A}(X)}, \neg \Omega_{\mathrm{MC}_{\mathcal B}(X),\mathbf 0})\\
        &\equiv \mathsf G_1(\mathscr{T}_{\mathrm{MC}_{\mathcal A}(X)}, \mathrm{CD}_{\mathrm{MC}_{\mathcal B}(X)}).
    \end{align*}
\end{theorem}
\begin{proof}
    Let \(\pi_1 : \mathrm{MC}(X) \times \mathcal A \times \mathbb R \to \mathrm{MC}(X)\),
    \(\pi_2 : \mathrm{MC}(X) \times \mathcal A \times \mathbb R \to \mathcal A\),
    and \(\pi_3 : \mathrm{MC}(X) \times \mathcal A \times \mathbb R \to \mathbb R\) be the
    standard coordinate projection maps.
    Define a choice function \(\gamma : \mathscr T_{\mathrm{MC}_{\mathcal A}(X)} \to \mathrm{MC}(X) \times \mathcal A \times \mathbb R\)
    so that \[[\pi_1(\gamma(W)); \pi_2(\gamma(W)) , \pi_3(\gamma(W))] \subseteq W.\]
    Let \(\Psi_{W} = \pi_1(\gamma(W))\), \(A_W = \pi_2(\gamma(W))\), and \(\varepsilon_W = \pi_3(\gamma(W))\).
    Then we define \(\overleftarrow{T}_{\mathrm{I},n} : \mathscr T_{\mathrm{MC}_{\mathcal A}(X)} \to \mathscr N_X[\mathcal A]\)
    by \(\overleftarrow{T}_{\mathrm{I},n}(W) = \mathscr N_X(A_W)\).

    We now define
    \(\overrightarrow{T}_{\mathrm{II},n} : \mathscr T_X \times \mathscr T_{\mathrm{MC}_{\mathcal A}(X)} \to \mathrm{MC}(X)\)
    in the following way.
    For \(A \in \mathcal A\) and \(U \in \mathscr N_X(A)\), let \(V_{A,U}\) be open so that
    \[A \subseteq V_{A,U} \subseteq \mathrm{cl}(V_{A,U}) \subseteq U.\]
    For \(W \in \mathscr T_{\mathrm{MC}_{\mathcal A}(X)}\) and \(U \in \overleftarrow{T}_{\mathrm{I},n}(W)\),
    define \(f_{W,U,n} : X \to \mathbb R\) by the rule
    \[f_{W,U,n}(x) =
    \begin{cases}
        n, & x \in \mathrm{cl}(X \setminus \mathrm{cl}(V_{A_W,U}));\\
        \max \Psi_W(x), & \mathrm{otherwise}.\\
    \end{cases}
    \]
    Then we set
    \[\overrightarrow{T}_{\mathrm{II},n}(U,W) =
    \begin{cases}
        \check{f}_{W,U,n}, & U \in \overleftarrow{T}_{\mathrm{I},n}(W);\\
        \mathbf 0, & \text{otherwise}.
    \end{cases}
    \]
    By Remark \ref{rmk:MaximumSelection} and Corollary \ref{cor:CuscoAgreeOpen},
    \(\overrightarrow{T}_{\mathrm{II},n}(U,W) \in \mathrm{MC}(X)\) and,
    if \(U \in \overleftarrow{T}_{\mathrm{I},n}(W)\),
    \[\overrightarrow{T}_{\mathrm{II},n}(U,W) \in [\Psi_W;A_W,\varepsilon_W] \subseteq W.\]

    Suppose we have a sequence
    \[\langle U_n : n \in \omega \rangle \in \prod_{n\in\omega} \overleftarrow{T}_{\mathrm{I},n}(W_n)\]
    for a sequence \(\langle W_n : n \in \omega \rangle\) of \(\mathscr T_{\mathrm{MC}_{\mathcal A}(X)}\)
    so that \(\{ U_n : n \in \omega \} \not\in \Lambda_X(\mathcal B)\).
    Let \(\Phi_n = \overrightarrow{T}_{\mathrm{II},n}(U_n,W_n)\) for each \(n \in \omega\).
    We can find \(N \in \omega\) and \(B \in \mathcal B\) so that, for every \(n \geq N\),
    \(B \not\subseteq U_n\).
    Now, suppose \(\Phi \in \mathrm{MC}(X) \setminus \{\Phi_n : n \in \omega\}\) is arbitrary.
    By Lemma \ref{lem:SequentiallyCompact}, \(\Phi[B]\) is bounded, so let \(M > \sup |\Phi[B]|\)
    and \(n \geq \max\{N,M+1\}\).
    Now, for \(x \in B \setminus U_n\), note that \(n \in \Phi_n(x)\) and that, for \(y \in \Phi(x)\),
    \[y \leq \sup |\Phi[B]| < M \leq n-1 \implies y-n < -1 \implies |y-n| > 1.\]
    In particular, \(\Phi_n(x) \not\subseteq \mathbb B(\Phi(x),1)\) which establishes that
    \(\Phi_n \not\in [\Phi; B, 1]\).
    Hence, \(\{\Phi_n : n \in \omega \}\) is closed and discrete and Theorem \ref{TranslationTheorem} applies.

    For what remains, observe that
    \[\mathsf G_1(\mathscr{T}_{\mathrm{MC}_{\mathcal A}(X)}, \mathrm{CD}_{\mathrm{MC}_{\mathcal B}(X)})
    \leq_{\mathrm{II}} \mathsf G_1(\mathscr T_{\mathrm{MC}_{\mathcal A}(X)}, \neg \Omega_{\mathrm{MC}_{\mathcal B}(X),\mathbf 0})\]
    since, if Two can produce a closed discrete set, then Two can avoid clustering around \(\mathbf 0\).
    Hence, by Corollary \ref{cor:Rothberger} we obtain that
    \begin{align*}
        \mathsf G_1(\mathscr{N}_X[\mathcal A], \neg \mathcal O_X(\mathcal B))
        &= \mathsf G_1(\mathscr{N}_X[\mathcal A], \neg \Lambda_X(\mathcal B))\\
        &\leq_{\mathrm{II}} \mathsf G_1(\mathscr{T}_{\mathrm{MC}_{\mathcal A}(X)}, \mathrm{CD}_{\mathrm{MC}_{\mathcal B}(X)})\\
        &\leq_{\mathrm{II}} \mathsf G_1(\mathscr T_{\mathrm{MC}_{\mathcal A}(X)}, \neg \Omega_{\mathrm{MC}_{\mathcal B}(X),\mathbf 0})\\
        &\equiv \mathsf G_1(\mathscr{N}_X[\mathcal A], \neg \mathcal O_X(\mathcal B)).
    \end{align*}
    This completes the proof.
\end{proof}

We now offer some relationships related to Gruenhage's \(W\)-games.
\begin{proposition} \label{prop:ConvergenceGames}
    Let \(\mathcal A\) and \(\mathcal B\) be ideals of closed subsets of \(X\).
    Then
    \begin{enumerate}[label=(\roman*)]
        \item \label{prop:ConvergenceGamesA}
        \(\mathsf G_1(\mathscr N_{\mathrm{MC}_{\mathcal A}(X),\mathbf 0} , \neg \Omega_{\mathrm{MC}_{\mathcal B}(X),\mathbf 0})
        \leq_{\mathrm{II}} \mathsf G_1(\mathscr N_{\mathrm{MC}_{\mathcal A}(X),\mathbf 0} , \neg \Gamma_{\mathrm{MC}_{\mathcal B}(X),\mathbf 0})\) and
        \item \label{prop:ConvergenceGamesB}
        \(\mathsf G_1(\mathscr N_{\mathrm{MC}_{\mathcal A}(X),\mathbf 0} , \neg \Gamma_{\mathrm{MC}_{\mathcal B}(X),\mathbf 0})
        \leq_{\mathrm{II}} \mathsf G_1(\mathscr N_X[\mathcal A], \neg \Gamma_X(\mathcal B))\).
    \end{enumerate}
\end{proposition}
\begin{proof}
    \ref{prop:ConvergenceGamesA} is evident since, if Two can avoid clustering at \(\mathbf 0\), they can surely
    avoid converging to \(\mathbf 0\).

    \ref{prop:ConvergenceGamesB}
    Fix \(U_0 \in \mathscr T_X\) and
    define \(\overleftarrow{T}_{\mathrm{I},n} : \mathscr N_X[\mathcal A] \to \mathscr N_{\mathrm{MC}_{\mathcal A}(X),\mathbf 0}\)
    by \(\overleftarrow{T}_{\mathrm{I},n}(\mathscr N_X(A)) = \left[\mathbf 0; A , 2^{-n}\right]\).
    Then define \(\overrightarrow{T}_{\mathrm{II},n} : \mathrm{MC}(X) \times \mathscr N_X[\mathcal A] \to \mathscr T_X\) by
    \(\overrightarrow{T}_{\mathrm{II},n}(\Phi, \mathscr N_X(A)) = \Phi^{\leftarrow}\left[\left(-2^{-n},2^{-n}\right)\right]\).
    Note that, if \(\Phi \in \left[\mathbf 0; A , 2^{-n}\right] = \overleftarrow{T}_{\mathrm{I},n}(\mathscr N_X(A)),\)
    then \(A \subseteq \Phi^\leftarrow\left[\left(-2^{-n},2^{-n}\right)\right]\), which establishes that
    \(\overrightarrow{T}_{\mathrm{II},n}(\Phi, \mathscr N_X(A)) \in \mathscr N_X(A)\).

    Suppose we have
    \[\left\langle \Phi_n : n \in \omega \right\rangle \in \prod_{n\in\omega} \overleftarrow{T}_{\mathrm{I},n}(\mathscr N_X(A_n))\]
    for a sequence \(\langle A_n : n \in \omega \rangle\) of \(\mathcal A\)
    so that \(\langle \Phi_n : n \in \omega \rangle \not\in \Gamma_{\mathrm{MC}_{\mathcal B}(X),\mathbf 0}\).
    Then we can find \(B \in \mathcal B\), \(\varepsilon > 0\), and \(N \in \omega\) so that \(2^{-N} < \varepsilon\) and,
    for all \(n \geq N\), \(\Phi_n \not\in [\mathbf 0; B, \varepsilon]\).

    To finish this application of Theorem \ref{TranslationTheorem}, we need to show that \[B \not\subseteq
    \overrightarrow{T}_{\mathrm{II},n}(\Phi_n, \mathscr N_X(A_n))\] for all \(n \geq N\).
    So let \(n \geq N\) and note that, since \(\Phi_n \not\in [\mathbf 0; B, \varepsilon]\),
    there is some \(x \in B\) and \(y \in \Phi_n(x)\) so that \(|y| \geq \varepsilon > 2^{-N} \geq 2^{-n}\).
    That is, \(\Phi_n(x) \not\subseteq (-2^{-n} , 2^{-n})\) and so
    \(x \not\in \overrightarrow{T}_{\mathrm{II},n}(\Phi_n, \mathscr N_X(A_n))\).
    This finishes the proof.
\end{proof}

Though particular applications of Corollary \ref{cor:Rothberger} and Theorem \ref{thm:SecondTheorem}
abound, we record a few that capture the general spirit using ideals of usual interest
after recalling some other facts and some names for particular selection principles.
\begin{definition}
    We identify some particular selection principles by name.
    \begin{itemize}
        \item
        \(\mathsf S_1(\Omega_{X,x}, \Omega_{X,x})\) is known as the \emph{strong countable fan-tightness property
        for \(X\) at \(x\)}.
        \item
        \(\mathsf S_1(\mathscr D_{X}, \Omega_{X,x})\) is known as the \emph{strong countable dense fan-tightness property
        for \(X\) at \(x\)}.
        \item
        \(\mathsf S_1(\mathscr T_X, \mathrm{CD}_X)\) is known as the \emph{discretely selective property for \(X\)}.
        \item
        We refer to \(\mathsf S_1(\Omega_X,\Omega_X)\) as the \emph{\(\omega\)-Rothberger property}
        and \(\mathsf S_1(\mathcal K_X,\mathcal K_X)\) as the \emph{\(k\)-Rothberger property}.
    \end{itemize}
\end{definition}

\begin{definition}
    For a partially ordered set \((\mathbb P, \leq)\) and collections \(\mathcal A , \mathcal B \subseteq \mathbb P\)
    so that, for every \(B \in \mathcal B\), there exists some \(A \in \mathcal A\) with \(B \subseteq A\),
    we define the \emph{cofinality of \(\mathcal A\) relative to \(\mathcal B\)} by
    \[\mathrm{cof}(\mathcal A; \mathcal B, \leq)
    = \min \{ \kappa \in \mathrm{CARD} : (\exists \mathscr F \in [\mathcal A]^\kappa)(\forall B \in \mathcal B)(\exists A \in \mathscr F)
    \ B \subseteq A\}\]
    where \(\mathrm{CARD}\) is the class of cardinals
    and \([\mathcal A]^{\kappa}\) is the set of \(\kappa\)-sized subsets of \(\mathcal A\).
\end{definition}
\begin{lemma} \label{lem:TelgarskyCofinal}
    Let \(\mathcal A, \mathcal B \subseteq \wp^+(X)\) for a space \(X\).

    As long as \(X\) is \(T_1\),
    \[\mathrm{I} \uparrow_{\mathrm{pre}} \mathsf G_1(\mathscr N_X[\mathcal A] , \neg \mathcal O_X(\mathcal B))
    \iff \mathrm{cof}(\mathcal A;\mathcal B, \subseteq) \leq \omega.\]
    (See {\cite{GerlitsNagy,TkachukCpBook}, and \cite[Lemma 23]{CHContinuousFunctions}}.)

    If \(\mathcal A\) consists of \(G_\delta\) sets,
    \begin{align*}
        \mathrm{I} \uparrow \mathsf G_1(\mathscr N_X[\mathcal A] , \neg \mathcal O_X(\mathcal B))
        &\iff \mathrm{I} \uparrow_{\mathrm{pre}} \mathsf G_1(\mathscr N_X[\mathcal A] , \neg \mathcal O_X(\mathcal B))\\
        &\iff \mathrm{cof}(\mathcal A;\mathcal B, \subseteq) \leq \omega.
    \end{align*}
    (See {\cite{Galvin1978,Telgarsky1975} and \cite[Lemma 24]{CHContinuousFunctions}}.)
\end{lemma}
Observe that Lemma \ref{lem:TelgarskyCofinal} informs us that, for a \(T_1\) space \(X\),
\begin{itemize}
    \item
    \(\mathrm{I} \uparrow_{\mathrm{pre}} \mathsf G_1(\mathbb P_X, \neg\mathcal O_X)\) if and only if \(X\) is countable,
    \item
    \(\mathrm{I} \uparrow_{\mathrm{pre}} \mathsf G_1(\mathscr N_X[K(X)], \neg\mathcal O_X)\) if and only if
    \(X\) is \(\sigma\)-compact, and
    \item
    \(\mathrm{I} \uparrow_{\mathrm{pre}} \mathsf G_1(\mathscr N_X[K(X)], \neg\mathcal K_X)\) if and only if \(X\) is hemicompact.
\end{itemize}
\begin{corollary}
    For an ideal \(\mathcal A\) of closed subsets of a \(T_1\) space \(X\), \(\mathrm{cof}(\mathcal A; \mathcal A, \subseteq) \leq \omega\)
    if and only if \(\mathrm{MC}_{\mathcal A}(X)\) is metrizable.
\end{corollary}
\begin{proof}
    If \(\{ A_n : n \in \omega \} \subseteq \mathcal A\) is so that, for every \(A \in \mathcal A\), there is an \(n\in\omega\)
    with \(A \subseteq A_n\), notice that the family \(\{ \mathbf W(A_n,2^{-m}) : n,m \in \omega\}\) is a countable
    base for the uniformity on \(\mathrm{MC}_{\mathcal A}(X)\); so Theorem \ref{thm:UniformMetrizability} demonstrates
    that \(\mathrm{MC}_{\mathcal A}(X)\) is metrizable.

    Now, suppose \(\mathrm{MC}_{\mathcal A}(X)\) is metrizable, which implies that \(\mathrm{MC}_{\mathcal A}(X)\) is first-countable.
    Using a descending countable basis at \(\mathbf 0\), we see that
    \[\mathrm{I} \uparrow_{\mathrm{pre}} \mathsf G_1(\mathscr N_{\mathrm{MC}_{\mathcal A}(X),\mathbf 0} ,
    \neg \Gamma_{\mathrm{MC}_{\mathcal A}(X),\mathbf 0}),\]
    and, in particular,
    \[\mathrm{I} \uparrow_{\mathrm{pre}} \mathsf G_1(\mathscr N_{\mathrm{MC}_{\mathcal A}(X),\mathbf 0} ,
    \neg \Omega_{\mathrm{MC}_{\mathcal A}(X),\mathbf 0}).\]
    By Theorem \ref{thm:SecondTheorem}, we see that
    \[\mathrm{I} \uparrow_{\mathrm{pre}} \mathsf G_1(\mathscr N_X[\mathcal A], \neg \mathcal O_X(\mathcal A)).\]
    So, by Lemma \ref{lem:TelgarskyCofinal}, \(\mathrm{cof}(\mathcal A; \mathcal A, \subseteq) \leq \omega\).
\end{proof}
As a particular consequence of this, we see that
\begin{corollary} \label{cor:ParticularMetrizability}
    For any regular space \(X\), the following are equivalent.
    \begin{enumerate}[label=(\roman*)]
        \item
        \(X\) is countable.
        \item
        \(\mathrm{MC}_p(X)\) is metrizable.
        \item
        \(\mathrm{MC}_p(X)\) is not discretely selective.
        \item
        \(\mathrm{II} \uparrow_{\mathrm{mark}} \mathsf G_1(\Omega_X,\Omega_X)\).
        \item
        \(\mathrm{II} \uparrow_{\mathrm{mark}} \mathsf G_1(\Omega_{\mathrm{MC}_p(X),\mathbf 0},\Omega_{\mathrm{MC}_p(X),\mathbf 0})\).
        \item
        \(\mathrm{II} \uparrow_{\mathrm{mark}} \mathsf G_1(\mathscr D_{\mathrm{MC}_p(X)},\Omega_{\mathrm{MC}_p(X),\mathbf 0})\).
    \end{enumerate}
    Also, the following are equivalent.
    \begin{enumerate}[label=(\roman*)]
        \item
        \(X\) is hemicompact.
        \item
        \(\mathbb K(X)\) is hemicompact. (See \cite[Theorem 3.22]{CHCompactOpen}.)
        \item
        \(\mathrm{MC}_k(X)\) is metrizable. (See \cite[Cor 4.5]{HolaHoly2023}.)
        \item
        \(\mathrm{MC}_k(X)\) is not discretely selective.
        \item
        \(\mathrm{II} \uparrow_{\mathrm{mark}} \mathsf G_1(\mathcal K_X,\mathcal K_X)\).
        \item
        \(\mathrm{II} \uparrow_{\mathrm{mark}} \mathsf G_1(\Omega_{\mathrm{MC}_k(X),\mathbf 0},\Omega_{\mathrm{MC}_k(X),\mathbf 0})\).
        \item
        \(\mathrm{II} \uparrow_{\mathrm{mark}} \mathsf G_1(\mathscr D_{\mathrm{MC}_k(X)},\Omega_{\mathrm{MC}_k(X),\mathbf 0})\).
    \end{enumerate}
\end{corollary}
\begin{theorem}[See {\cite[Cor 11]{CHContinuousFunctions}} and {\cite{TkachukCpBook}}] \label{thm:TkachukStrengthening}
    Let \(\mathcal A\) and \(\mathcal B\) be ideals of closed subsets of \(X\).
    Then
    \[\mathrm{I} \uparrow \mathsf G_1(\mathscr N_X[\mathcal A], \neg \mathcal O_X(\mathcal B))
    \iff \mathrm{I} \uparrow \mathsf G_1(\mathscr N_X[\mathcal A], \neg \Gamma_X(\mathcal B))\]
    and
    \[\mathrm{I} \uparrow_{\mathrm{pre}} \mathsf G_1(\mathscr N_X[\mathcal A], \neg \mathcal O_X(\mathcal B))
    \iff \mathrm{I} \uparrow_{\mathrm{pre}} \mathsf G_1(\mathscr N_X[\mathcal A], \neg \Gamma_X(\mathcal B))\]
\end{theorem}
\begin{corollary} \label{cor:WeakerThanMetrizability}
    For any regular space \(X\), the following are equivalent.
    \begin{enumerate}[label=(\roman*)]
        \item
        \(\mathrm{II} \uparrow \mathsf G_1(\Omega_X,\Omega_X)\).
        \item
        \(\mathrm{II} \uparrow \mathsf G_1(\Omega_{\mathrm{MC}_p(X),\mathbf 0},\Omega_{\mathrm{MC}_p(X),\mathbf 0})\).
        \item
        \(\mathrm{II} \uparrow \mathsf G_1(\mathscr D_{\mathrm{MC}_p(X)},\Omega_{\mathrm{MC}_p(X),\mathbf 0})\).
        \item
        \(\mathrm{I} \uparrow \mathsf G_1(\mathscr T_{\mathrm{MC}_p(X)}, \mathrm{CD}_{\mathrm{MC}_p(X)})\).
        \item
        \(\mathrm{I} \uparrow \mathsf G_1(\mathscr N_X[[X]^{<\omega}], \neg \Omega_X)\).
        \item
        \(\mathrm{I} \uparrow \mathsf G_1(\mathscr N_X[[X]^{<\omega}], \neg \Gamma_\omega(X))\).
    \end{enumerate}
    Also, the following are equivalent.
    \begin{enumerate}[label=(\roman*)]
        \item
        \(\mathrm{II} \uparrow \mathsf G_1(\mathcal K_X,\mathcal K_X)\).
        \item
        \(\mathrm{II} \uparrow \mathsf G_1(\Omega_{\mathrm{MC}_k(X),\mathbf 0},\Omega_{\mathrm{MC}_k(X),\mathbf 0})\).
        \item
        \(\mathrm{II} \uparrow \mathsf G_1(\mathscr D_{\mathrm{MC}_k(X)},\Omega_{\mathrm{MC}_k(X),\mathbf 0})\).
        \item
        \(\mathrm{I} \uparrow \mathsf G_1(\mathscr T_{\mathrm{MC}_k(X)}, \mathrm{CD}_{\mathrm{MC}_k(X)})\).
        \item
        \(\mathrm{I} \uparrow \mathsf G_1(\mathscr N_X[K(X)], \neg \mathcal K_X)\).
        \item
        \(\mathrm{I} \uparrow \mathsf G_1(\mathscr N_X[K(X)], \neg \Gamma_k(X)\).
    \end{enumerate}
\end{corollary}
In general, Corollaries \ref{cor:ParticularMetrizability} and \ref{cor:WeakerThanMetrizability} are strictly separate,
as the following example demonstrates.
\begin{example}
    Let \(X\) be the one-point Lindel{\"{o}}fication of \(\omega_1\) with the discrete topology,
    an instance of a Fortissimo space \cite[Space 25]{Counterexamples}.
    In \cite[Example 3.24]{CHCompactOpen}, it is shown that \(X\) has the property that
    \(\mathrm{II} \uparrow \mathsf G_1(\mathcal K_X,\mathcal K_X)\), but
    \(\mathrm{II} {\not\uparrow}_{\mathrm{mark}} \mathsf G_1(\mathcal K_X,\mathcal K_X)\).
    Since the compact subsets of \(X\) are finite, we see also that
    \(\mathrm{II} \uparrow \mathsf G_1(\Omega_X,\Omega_X)\), but
    \(\mathrm{II} {\not\uparrow}_{\mathrm{mark}} \mathsf G_1(\Omega_X,\Omega_X)\).
\end{example}

However, according to Theorem \ref{thm:PawlikowskiStuff}, if Two can win against predetermined strategies
in some Rothberger-like games, Two can actually win against full-information strategies in those games.
\begin{theorem} \label{thm:PawlikowskiStuff}
    Let \(X\) be any space.
    \begin{enumerate}[label=(\roman*)]
        \item
        By Pawlikowski \cite{Pawlikowski},
        \[\mathrm{I} \uparrow_{\mathrm{pre}} \mathsf G_1(\mathcal O_X , \mathcal O_X) \iff
        \mathrm{I} \uparrow \mathsf G_1(\mathcal O_X , \mathcal O_X).\]
        \item
        By Scheepers \cite{ScheepersIII} (see also \cite[Cor 4.12]{CHVietoris}),
        \[\mathrm{I} \uparrow_{\mathrm{pre}} \mathsf G_1(\Omega_X , \Omega_X) \iff
        \mathrm{I} \uparrow \mathsf G_1(\Omega_X , \Omega_X).\]
        \item
        By \cite[Theorem 4.21]{CHVietoris},
        \[\mathrm{I} \uparrow_{\mathrm{pre}} \mathsf G_1(\mathcal K_X , \mathcal K_X) \iff
        \mathrm{I} \uparrow \mathsf G_1(\mathcal K_X , \mathcal K_X).\]
    \end{enumerate}
\end{theorem}
\begin{corollary} \label{cor:BigRothberger}
    For any regular space \(X\), the following are equivalent.
    \begin{enumerate}[label=(\roman*)]
        \item
        \(X\) is \(\omega\)-Rothberger.
        \item
        \(X^{<\omega}\) is Rothberger, where \(X^{<\omega}\) is the disjoint union of \(X^n\) for all \(n \geq 1\).
        (See \cite{Sakai1988} and \cite[Cor 3.11]{CHVietoris}.)
        \item
        \(\mathcal{P}_{\mathrm{fin}}(X)\) is Rothberger, where \(\mathcal{P}_{\mathrm{fin}}(X)\)
        is the set \([X]^{<\omega}\) with the subspace topology inherited from \(\mathbb K(X)\).
        (See \cite[Cor 4.11]{CHVietoris}.)
        \item
        \(\mathrm{I} \not\uparrow \mathsf G_1(\Omega_X , \Omega_X)\).
        \item
        \(\mathrm{MC}_p(X)\) has strong countable fan-tightness at \(\mathbf 0\).
        \item
        \(\mathrm{MC}_p(X)\) has strong countable dense fan-tightness at \(\mathbf 0\).
        \item
        \(\mathrm{II} {\not\uparrow}_{\mathrm{mark}} \mathsf G_1(\mathscr N_X[[X]^{<\omega}], \neg \Omega_X)\).
        \item
        \(\mathrm{II} \not\uparrow \mathsf G_1(\mathscr N_X[[X]^{<\omega}], \neg \Omega_X)\).
        \item
        \(\mathrm{II} {\not\uparrow}_{\mathrm{mark}} \mathsf G_1(\mathscr T_{\mathrm{MC}_p(X)}, \mathrm{CD}_{\mathrm{MC}_p(X)})\).
        \item
        \(\mathrm{II} \not\uparrow \mathsf G_1(\mathscr T_{\mathrm{MC}_p(X)}, \mathrm{CD}_{\mathrm{MC}_p(X)})\).
        \item
        \(\mathrm{II} {\not\uparrow}_{\mathrm{mark}} \mathsf G_1(\mathscr N_{\mathrm{MC}_p(X),\mathbf 0}, \neg \Omega_{\mathrm{MC}_p(X),\mathbf 0})\).
        \item
        \(\mathrm{II} \not\uparrow \mathsf G_1(\mathscr N_{\mathrm{MC}_p(X),\mathbf 0}, \neg \Omega_{\mathrm{MC}_p(X),\mathbf 0})\).
    \end{enumerate}
    Also, the following are equivalent.
    \begin{enumerate}[label=(\roman*)]
        \item
        \(X\) is \(k\)-Rothberger.
        \item
        \(\mathrm{I} \not\uparrow \mathsf G_1(\mathcal K_X , \mathcal K_X)\).
        \item
        \(\mathrm{MC}_k(X)\) has strong countable fan-tightness at \(\mathbf 0\).
        \item
        \(\mathrm{MC}_k(X)\) has strong countable dense fan-tightness at \(\mathbf 0\).
        \item
        \(\mathrm{II} {\not\uparrow}_{\mathrm{mark}} \mathsf G_1(\mathscr N_X[K(X)], \neg \mathcal K_X)\).
        \item
        \(\mathrm{II} \not\uparrow \mathsf G_1(\mathscr N_X[K(X)], \neg \mathcal K_X)\).
        \item
        \(\mathrm{II} {\not\uparrow}_{\mathrm{mark}} \mathsf G_1(\mathscr T_{\mathrm{MC}_k(X)}, \mathrm{CD}_{\mathrm{MC}_k(X)})\).
        \item
        \(\mathrm{II} \not\uparrow \mathsf G_1(\mathscr T_{\mathrm{MC}_k(X)}, \mathrm{CD}_{\mathrm{MC}_k(X)})\).
        \item
        \(\mathrm{II} {\not\uparrow}_{\mathrm{mark}} \mathsf G_1(\mathscr N_{\mathrm{MC}_k(X),\mathbf 0}, \neg \Omega_{\mathrm{MC}_k(X),\mathbf 0})\).
        \item
        \(\mathrm{II} \not\uparrow \mathsf G_1(\mathscr N_{\mathrm{MC}_k(X),\mathbf 0}, \neg \Omega_{\mathrm{MC}_k(X),\mathbf 0})\).
    \end{enumerate}
\end{corollary}

We end this section with a couple examples that demonstrate the difference between these two classes
of results.
\begin{example}
    The reals \(\mathbb R\) are hemicompact (and thus \(k\)-Rothberger) but not \(\omega\)-Rothberger.
    Indeed, since every \(\omega\)-Rothberger space is Rothberger by \cite{Sakai1988} and \(\mathbb R\)
    is not Rothberger, \(\mathbb R\) is not \(\omega\)-Rothberger.
    In particular, \(\mathrm{MC}_k(\mathbb R)\) is metrizable whereas \(\mathrm{MC}_p(\mathbb R)\)
    does not even have countable strong fan-tightness at \(\mathbf 0\).
\end{example}
\begin{example}
    The rationals \(\mathbb Q\) are countable (and thus \(\omega\)-Rothberger) but, by \cite[Prop 5]{AppskcoversII},
    not \(k\)-Rothberger.
    In particular, \(\mathrm{MC}_p(\mathbb Q)\) is metrizable whereas \(\mathrm{MC}_k(\mathbb Q)\) does not
    have countable strong fan-tightness at \(\mathbf 0\).
\end{example}

\section{Obstacles to generalization}

In this section, we discuss difficulties that arise in trying to extend Theorems \ref{thm:HolaHolyUscoChar} and \ref{thm:HolaHolyCuscoChar} to classes of usco maps other than minimal usco and cusco maps.

First, notice that not every selection of a minimal cusco map is quasicontinuous and subcontinuous, as suggested
by Theorem \ref{thm:HolaHolyCuscoChar}.
Since every usco map contains a minimal usco map, every usco map has some quasicontinuous and subcontinuous
selection.
However, the key to making sure that those selections from a minimal cusco map \(\Phi\)
bring one back to \(\Phi\) via closure and convex hull, as in Theorem \ref{thm:HolaHolyCuscoChar}, is convexity.
Without an analogous structure in place, there is no clear way to make sure a correspondence
of this type holds for other classes of usco maps.

Inspired by the operation of the convex hull, one my think similar additions may generate
interesting examples.
However, adding points to the vertical sections of a minimal usco map may create a graph which is not closed.
\begin{example} \label{example:FirstExample}
    Let \(a_n = \frac{n}{n+1}\), \(b_n = \frac{n+1}{n+2}\), and
    \(\mathrm{mid}_n = \frac{a_n + b_n}{2}\).
    Define the funtion \(f:(0,1) \to [0,2]\) by
    \[
    f(x)
    =\begin{cases}
        \frac{2}{b_n - a_n}(x - a_n) & x \in (a_n,\mathrm{mid}_n) \\
        2 & x \in [\mathrm{mid}_n,b_n] \\
    \end{cases}
    \]
    Then \(f\) is quasicontinuous and subcontinuous by Lemma \ref{lem:UsefulFunction}.
    Thus, \(\overline{f}\) is minimal usco.
    However,
    \[
    G = \overline{f}  \cup \left\{\left(x, \frac{\max\overline{f}(x) + \min\overline{f}(x)}{2} \right) : x \in [0,1] \right\},
    \]
    the map created by adding the midpoint of each vertical section, is not usco.
    Indeed, let \(\varepsilon > 0\) be small enough so that \(\frac{3}{2} \notin W := (-\varepsilon, 1+\varepsilon) \cup (2 - \varepsilon, 2 + \varepsilon)\). Then \(1 \in G^{\leftarrow}(W)\), but for every \(N\), there is an \(x \in (1 - \frac{1}{N},1]\) so that \(\frac{3}{2} \in G(x)\); thus \(x \notin G^{\leftarrow}(W)\), meaning this preimage is not open, and therefore \(G\) is not usco.
    The graphs of \(\overline{f}\) and \(G\) are in Figure \ref{fig:FirstExample}.
\end{example}
\begin{figure}
    \begin{tikzpicture}[scale=3.5]
        \draw [->] (-0.2,0) -- (1.2,0) node [below right] {\(x\)};
        \draw [->] (0,-0.2) -- (0,2.2) node [above right] {\(y\)};
        \foreach \n in {0,1,...,50}{
            \draw ({\n/(\n+1)},0) -- ({(1/2)*(\n/(\n+1) + (\n+1)/(\n+2))},1);
            \draw ({(1/2)*(\n/(\n+1) + (\n+1)/(\n+2))},2) -- ({(\n+1)/(\n+2)},2);
        }
        \draw (1,0) -- (1,1);
        \draw [fill=black] (1,2) circle (0.1pt);
    \end{tikzpicture}
    \begin{tikzpicture}[scale=3.5]
        \draw [->] (-0.2,0) -- (1.2,0) node [below right] {\(x\)};
        \draw [->] (0,-0.2) -- (0,2.2) node [above right] {\(y\)};
        \foreach \n in {0,1,...,50}{
            \draw ({\n/(\n+1)},0) -- ({(1/2)*(\n/(\n+1) + (\n+1)/(\n+2))},1);
            \draw [fill=black] ({(1/2)*(\n/(\n+1) + (\n+1)/(\n+2))},{3/2}) circle (0.1pt);
            \draw ({(1/2)*(\n/(\n+1) + (\n+1)/(\n+2))},2) -- ({(\n+1)/(\n+2)},2);
        }
        \draw (1,0) -- (1,1);
        \draw [fill=black] (1,2) circle (0.1pt);
    \end{tikzpicture}
    \caption{Adding midpoints}
    \label{fig:FirstExample}
\end{figure}
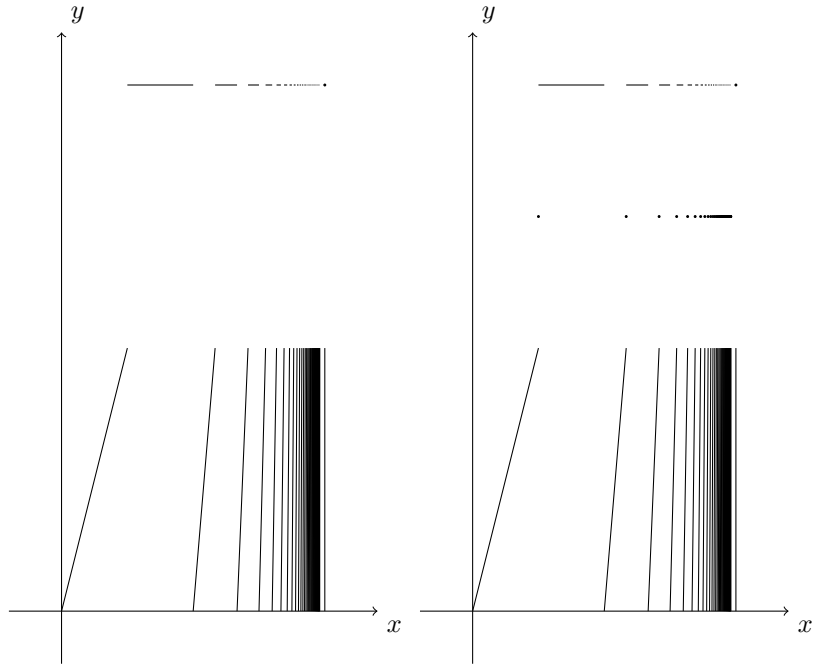

Note, however, by \cite{Christensen1982}, we can take the closure \(\overline{G}\) of the \(G\) defined in Example
\ref{example:FirstExample} to create an usco map since the resulting graph will be contained in the
graph of a cusco map.
Unfortunately, it's not clear what kind of structural facts could be used to ensure that quasicontinuous and
subcontinuous selections generate the same given map.
For example, one could add the singleton \(\{0\}\) to every section when mapping into the reals and, in the end,
there would be maps with distinct selections that are quasicontinuous and subcontinuous but that don't generate
the original map with the given procedure.
In the convex setting, one can use half-spaces to separate compact convex sets from convex sets to eventually
arrive at Theorem \ref{thm:HolaHolyCuscoChar}.
These tools offered by the convex setting do not adapt to the operation of adding midpoints as described above.

More broadly, we would like to find maps that complete the commuting diagram in Figure \ref{fig:Commuting}, but is not clear at this time if anything other than minimal usco (where \(\mathfrak m\) is the closure and \(\mathfrak e\) is the identity) and minimal cusco (where \(\mathfrak m\) is the pointwise convex hull of the closure and \(\mathfrak e\) is the pointwise convex hull) maps work.
\begin{figure}
    \begin{tikzpicture}
        \node (Q) at (0,2) {\(\mathcal{QS}(X,Y)\)};
        \node (M) at (0,0) {\(\mathrm{MU}(X,Y)\)};
        \node (U) at (4,0) {\(\mathrm{USCO}(X,Y)\)};
        \node (RM) at (0,-2) {\(\mathrm{MU}(U,Y)\)};
        \node (RU) at (4,-2) {\(\mathrm{USCO}(U,Y)\)};
        \draw [->] (Q) -- (M) node [midway,left] {\(\mathrm{cl}{\,}\mathrm{gr}\)};
        \draw [->] (Q) -- (U) node [midway,above] {\(\mathfrak{m}\)};
        \draw [->,dashed] (M) -- (U) node [midway,above] {\(\mathfrak e\)};
        \draw [->] (M) -- (RM) node [midway,left] {\(\restriction_U\)};
        \draw [->] (U) -- (RU) node [midway,right] {\(\restriction_U\)};
        \draw [->,dashed] (RM) -- (RU) node [midway,above] {\(\mathfrak e_U\)};
    \end{tikzpicture}
    \caption{Commuting diagram}
    \label{fig:Commuting}
\end{figure}
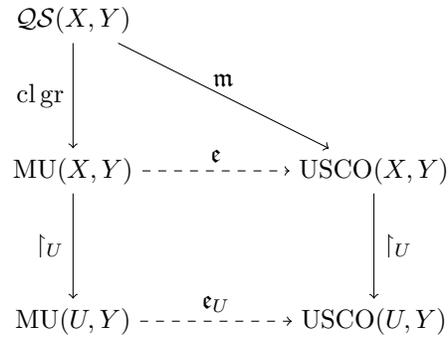

\section{Questions} \label{sec:Questions}

We end with a few questions.

\begin{question} \label{question:Generality}
    As mentioned after Corollary \ref{cor:Rothberger}, many of the equivalences here can be expanded
    to include \(\mathrm{MU}_{\mathcal A}(X)\) and, under the additional assumption that \(X\) is
    functionally \(\mathcal A\)-normal, even the appropriately topologized ring \(C_{\mathcal A}(X)\)
    of continuous real-valued function from \(X\).
    Is there a more general theory, perhaps relative to the set of quasicontinuous and subcontinuous real-valued
    functions, that unifies all of these results?
\end{question}
\begin{question}
    In general, one could define an operator \(\mathcal C : K(Y) \to K(Y)\) to apply to the outputs of minimal usco
    maps.
    In the cusco case, \(\mathcal C\) is the convex hull.
    For what kind of operators \(\mathcal C\) do we obtain analogues to Theorems \ref{thm:HolaHolyUscoChar} and
    \ref{thm:HolaHolyCuscoChar}?
\end{question}
\begin{question} \label{question:OtherFunctions}
    Are there other maps \(\mathfrak m\) as in Figure \ref{fig:Commuting} that make the diagram commute?
\end{question}
\begin{question} \label{question:Homogeneous}
    Can results similar to Theorems \ref{thm:FirstEquivalence} and \ref{thm:SecondTheorem} be established
    relative to \(\Omega_{\mathrm{MC}_{\mathcal A}(X),\Phi}\) for any \(\Phi \in \mathrm{MC}(X)\)?
\end{question}
 \begin{question}
    How many of the equivalences and dualities of this paper can be established for games of longer length
    and for finite-selection games?
 \end{question}
\begin{question}
    How much of this theory can be recovered when we study \(\mathrm{MC}(X,Y)\) for \(Y \neq \mathbb R\), for example,
    when \(Y\) is \([0,1]\), \(Y = \mathbb S^1\) (the circle group), or \(Y\) is a general Hausdorff locally convex linear space?
\end{question}

\providecommand{\bysame}{\leavevmode\hbox to3em{\hrulefill}\thinspace}
\providecommand{\MR}{\relax\ifhmode\unskip\space\fi MR }
\providecommand{\MRhref}[2]{%
  \href{http://www.ams.org/mathscinet-getitem?mr=#1}{#2}
}
\providecommand{\href}[2]{#2}

\end{document}